\documentclass[pdftex]{amsart}

\usepackage{amsmath}
\usepackage{amssymb}
\usepackage{amscd}
\usepackage{epsfig}
\usepackage{xcolor}
\usepackage{setspace}
\usepackage{verbatim}
\usepackage{diagrams} 
\usepackage{hyperref}

\newtheorem{prop}{Proposition} 
\newtheorem{lem}{Lemma} 
\newtheorem{cor}{Corollary}
\newtheorem{thm}{Theorem}

\newtheorem{conj}{Conjecture}

\DeclareMathOperator{\OO}{\mathcal{O}}
\DeclareMathOperator{\PP}{\mathbb{P}}

\DeclareMathOperator{\K}{\mathcal{K}}
\DeclareMathOperator{\F}{\mathcal{F}}
\DeclareMathOperator{\G}{\mathcal{G}}
\DeclareMathOperator{\LL}{\mathcal{L}}
\DeclareMathOperator{\C}{\mathbb{C}}
\DeclareMathOperator{\Hom}{\mathrm{Hom}}

\DeclareMathOperator{\Fl}{\mathrm{Fl}}
\DeclareMathOperator{\supp}{\mathrm{supp}}
\DeclareMathOperator{\Quot}{\mathrm{Quot}}
\DeclareMathOperator{\Map}{\mathrm{Map}}

\DeclareMathOperator{\HH}{\mathcal{H}}

\begin{document}

\title{K-theoretic J-functions of Type A Flag Varieties}
\author{K. Taipale}

\begin{abstract} The J-function in Gromov-Witten theory is a generating function for one-point genus zero Gromov-Witten invariants with descendants. Here we give formulas for the quantum K-theoretic J-functions of type A flag manifolds and conjectural formulas for other types. Some K-theoretic tools for computation are also provided. As an application, we prove the quantum K-theoretic J-function version of the abelian-nonabelian correspondence for Grassmannians and products of projective space.
\end{abstract}

\maketitle

\tableofcontents
\section{Introduction and preliminaries}

The J-function of a space $X$ is a generating function for one-point genus zero Gromov-Witten invariants with descendants. The (cohomological) J-function was introduced by Givental in solving the quantum differential equation \[ \hbar \frac{\partial F}{\partial t_i} = \partial_i \star F,\] where $\hbar$ is a constant, the $t_i$ are coordinates for cohomology of $X$, and $\partial_i:=\frac{\partial}{\partial t_i}$ \cite{GivEq}. For us, the operation $\star$ refers to multiplication in the small quantum cohomology ring. In ``nice'' cases the J-function compactly encodes the information necessary for reconstruction of Gromov-Witten theory in genus zero (\cite{LP}). This makes the J-function useful for relating Gromov-Witten theories of different spaces. In 1998, Givental used this cohomological J-function to prove the Mirror Theorem (\cite{GToric}). In addition, the abelian-nonabelian correspondence can be phrased naturally using the J-function (\cite{BCK1}). 

Quantum K-theory is a relatively underexplored extension of Gromov-Witten theory introduced by Givental and Lee (\cite{GL, LeeF,GWDVV}). (More recent foundational work can be found in \cite{TrueGiv}.) A K-theoretic Gromov-Witten invariant is the Euler characteristic of a sheaf obtained by pulling vector bundles from $X$ back to the moduli space of stable curves. Now the quantum differential equation solved looks like \[(1-e^{-\hbar}) \partial_{i} F = T_{i} \star F.\] Here the $T_i$ are classes in the Grothendieck group of $X$, rather than in the cohomology group, and $\star$ is quantum K-theoretic multiplication \cite{LeeF}. Recent work on quantum K-theory of Grassmannians can be found in \cite{BM, BCMP}.

This paper establishes K-theoretic J-functions for flag varieties of type A by using localization on Grothendieck's Quot scheme. The proofs use many techniques and results from Bertram, Ciocan-Fontanine, and Kim's work on the cohomological J-functions of flag varieties and Grassmannians and the abelian-nonabelian correspondence \cite{BCK1, BCK2}. The results extend Givental and Lee's ad hoc computations of the J-function for projective space in \cite{GL}. Along the way, we prove K-theoretic modifications of many cohomological tools (correspondence of residues, a pushforward lemma, and multiplicativity of the K-theoretic J-function). A K-theoretic J-function version of the abelian-nonabelian correspondence between Grassmannians and products of projective space is also proved.

\medskip

\paragraph{Acknowledgements}

I would like to thank my advisor Ionu\c{t} Ciocan-Fontanine for his mathematical generosity in introducing me to this area. He suggested the problem of the K-theoretic J-function to me, and many of these results appeared in my thesis.

\subsection{The J-function} \label{sec: jfunc}
In this paper we consider only homogeneous varieties. Definitions will be simplified with that in mind. In particular we need not consider virtual fundamental classes. 

The cohomological and K-theoretic J-functions of a homogeneous variety $X$ are defined as pushforwards of a residue on the graph space of stable maps to $X$ \cite{BAnother}. We describe the graph space first and the residue next, then discuss implications; see (\cite{GL}) for another description of the construction. 

The graph space $\overline{G}_{0,0}(X,\beta)$ is the space of stable maps $f:C \rightarrow X$, $C$ a genus zero curve with no marked points and image $f_*[C] = \beta \in H^*(X)$, with the requirement that $C$ has a distinguished parameterized component $C_0 \cong \PP^1$. (For flag varieties whose points are flags $V_1 \subset V_2 \subset \cdots V_{\ell}$, let $S_i$ denote the $i$th tautological bundle. Use the notation $d = (d_1, \ldots, d_{\ell})$ to represent the class in cohomology with degree $d_i$ intersection with $c_1(S_i^{\vee})$.) Recall that stable maps are at worst nodal.

Endow the distinguished parameterized component $C_0$ of $C$ with a $\C^*$-action. This lifts to a $\C^*$-action on the graph space and gives rise to $\C^*$-fixed loci in $\overline{G}_{0,0}(X,d)$ indexed by pairs $(d^+, d^-)$, $d^++d^- =d$. (Envision trees of $\PP^1$ concentrated on 0 and $\infty$ of the distinguished $\PP^1$.) The $\C^*$-fixed locus with the curve concentrated over $0 \in \PP^1$ is  $\overline{M}_{0,1}(X,d)$, the moduli space of degree $d$ stable maps with one marked point (the point at which the degree $d$ curve meets the distinguished $\PP^1$). In fact, we will think of $\hbar$ as the weight of the $\C^*$-action at zero in $\PP^1$. To simplify notation, let $M^X_d = \overline{M}_{0,1}(X,d)$ and $G^X_d = \overline{G}_{0,0}(X,d)$.

Let $ev$ denote the evaluation map taking the marked point $p_{1} $ on $C$ to $f(p_{1}) \in X$. Denote the cohomological Euler class of a vector bundle $N$ on $X$ by $e(N)$ and the K-theoretic Euler class by $\lambda_{-1}(N)$, and when necessary indicate $T$-equivariant Euler classes by a superscript $T$.

The cohomological and K-theoretic J-functions, then, are given by the following formal series: \begin{align*} J^X(t) &= \sum_{d=0}^{\infty} e^{dt} J_d^X(\hbar) \\ J^{X,K}(t) &= \sum_{d =0}^{\infty} e^{dt} J_d^{X,K}(e^{-\hbar}) \\ &= \sum_{d=0}^{\infty} Q^dJ_d^{X,K}(q).\end{align*} This version of the J-function, used in \cite{LP}, differs by a multiplicative factor from the version in \cite{BCK1, BCK2, GL}. Notice that we write $q = e^{-\hbar}$ and $Q^d= e^{d_1t_1 + d_2t_2 + \cdots +d_{\ell}t_{\ell}}$, where $(d_1, \ldots,  d_{\ell}) \in H_2(X)$. The coefficients of the series are the pushforwards mentioned earlier. The coefficients of the K-theoretic J-function are given by \begin{equation} J_d^{X,K} (e^{-\hbar}) = ev_* \left( \frac{[ \OO_{M^X_d}]}{\lambda_{-1}(N^{\vee}_{M^X_d/G^X_d})} \right).\end{equation} 


The cohomological J-function is a cohomology-valued generating function for one-point descendant Gromov-Witten invariants: \begin{displaymath} \int_X J^X_d \wedge \gamma = \sum_{a=0}^{\infty} \hbar^{-a-2} \langle \tau_a (\gamma) \rangle_d.\end{displaymath} Here $\gamma \in H^*(X)$ and $\langle \tau_a(\gamma) \rangle_d $ denotes the $a$th descendant Gromov-Witten invariant with degree $d$ \cite{CK}. This is a convergent series. In quantum K-theory, something similar is again true: \begin{align} \chi_X( J^{X,K}_d \otimes [\OO_{\gamma}]) &= \sum_{a=0}^{\infty} e^{-a \hbar} \chi(\overline{M}_{0,n}(X,d); [\OO_{\gamma}] \otimes L^a) \\ &=\sum_{a=0} q^a (\tau_k (\OO_{\gamma}))_{0,n,d}^K.\end{align} Here $L$ is the cotangent line bundle at the one marked point \cite{LP}. The J-function converges in the  $Q$-adic topology \cite{LeeF}.

We use localization to compute $J_d^{X,K}(q)$. $\PP^n_d$, the Drinfeld or quasimap compactification of $\mathrm{Map}_d(\PP^1,\PP^n)$, and Grothendieck's Quot scheme are alternative compactifications of $\mathrm{Map}_d(\PP^1,Gr(r,n))$ with more tractable fixed-point loci. Localization computations on $\overline{M}_{0,n}(X,d)$ transfer to $\PP^n_d$ or Quot via the correspondence of residues, described below. A $\C^*$-action on the rational curve $C$ induces a $\C^*$-action on any compactification of the map space. This $\C^*$ action will be used in localization computations.

\subsection{Quasimap spaces and the Quot scheme}

 $\PP^N_d$ is the quasimap compactification of $\mathrm{Map}_d(\PP^1,\PP^N)$. It is the space of $N+1$ bilinear forms $(f_0(x:y), \ldots, f_N(x:y))$. In general each $f_i$ is of degree $d$ and the $f_i$ are relatively prime, but to compactify the space we discard the condition that they be relatively prime. When $x^d$ is a common factor of all the $f_i$ we get a copy of $\PP^N$ that is fixed under the induced $\C^*$-action. 
 
The quasimap compactification of $\mathrm{Map}_d (\PP^1,Gr(r,n))$ space is obtained by looking at $Gr(r,n)$ embedded into $\PP^N$ via the Pl\"ucker embedding. Call this quasimap compactification $Gr(r,n)_d$. Then $Gr(r,n)$ is obtained as a $\C^*$-invariant locus of $Gr(r,n)_d$, just as $\PP^N$ is invariant in $\PP^N_d$.
 
 The quasimap space $\PP^N_d$ and the moduli spaces of stable maps to a projective variety $X \subset \PP^N$ are related by the following diagram:
\[\begin{CD}
G_d^{X} @>{\varphi}>> \PP^N_d \\
@AiAA   @AjAA \\
M_d^{X} @>{ev}>> X \\ 
\end{CD}.\]  

The quasimap space has rational singularities. The above relationship was exploited to get the cohomological and K-theoretic J-functions of projective space $\PP^n$.

The Quot scheme $\Quot_{\PP^1,d}(\C^n, n-r)$ is a smooth projective variety compactifying the map space $\mathrm{Map}_d(\PP^1, Gr(r,n))$, as long as $n > r$. It is the moduli space of degree $d$ quotient sheaves $Q$ with rank $n-r$: \[ 0 \rightarrow S \rightarrow \C^n \otimes \OO_{\PP^1} \rightarrow Q \rightarrow 0.\] $S$ is a locally free sheaf of rank $r$. The fixed point loci of the Quot scheme under various torus actions will be discussed in section \ref{sec: quot}.

We use orbifold pushforward to move the calculation of the residue to the Quot scheme. Since there is no map from the graph space to the Quot scheme, we rely on the fact that both the Quot scheme and the graph space have maps to the quasimap space. This aspect of the proof is discussed in section \ref{sec: simlem}.

\subsection{Localization and correspondence of residues}
 
Here we set notation and provide a brief overview of the localization techniques in K-theory used. An excellent introduction to equivariant K-theory can be found in Chapter 5 of Chriss and Ginzburg's book \cite{ChGz} and we direct the reader there for more.

 Symplectic and algebraic geometers have used localization in cohomology to obtain the Atiyah-Bott(-Berline-Vergne) residue theorem (\cite{EG, ES, Tu}) and the correspondence of residues \cite{LLY, BAnother}. Let $X$ be a nonsingular variety endowed with the action of a torus $(C^*)^n=T$. The ABBV formula allows us to shift integrals over $X$ to computations over $X^T$, the torus-fixed locus. A version of the ABBV formula for equivariant K-theory of smooth schemes is given in Chriss and Ginzburg (\cite{ChGz}) and virtual localization for K-theory on Deligne-Mumford stacks is discussed in \cite{RR3}. We present a version of the ABBV formula here to fix notation and justify later use.

The K-theoretic analogue of the Euler class of a bundle $N$ is \[\lambda_{-1}(N^{\vee}) = \sum (-1)^i \wedge^i N^{\vee} = 1- N^{\vee} + \wedge^2 N^{\vee} -... \] For $N$ an equivariant bundle, $\lambda_{-1}(N^{\vee})$ can also be considered equivariantly.

\begin{lem}[ABBV for K-theory, \cite{ChGz, RR3}] Let $X$ be a nonsingular scheme or smooth Deligne-Mumford stack endowed with the action of $T = (\C^*)^n$. Let $\alpha \in K(X)$ and let $i: W = \amalg W_k \hookrightarrow X$ be the inclusion of the components $W_k$ of $X^T$ into $X$, with restriction $i_k : W_k \rightarrow X$. Then \[\alpha = i_* \sum_k \frac{i_k^*(\alpha)}{\lambda_{-1}(N^{\vee}_{W_k/X})}.\]  \end{lem}

Correspondence of residues allows us to compare localization contributions of corresponding torus-fixed loci in the Quot scheme and the quasimap space. Proofs for correspondence of residues for complex manifolds and orbifolds were given by Lian-Liu-Yau and Bertram (\cite{LLY, BAnother}). 

Consider $X$ and $Y$ smooth varieties (or smooth Deligne-Mumford stacks) endowed with a $T$-action, and $g: X \rightarrow Y$ a proper $T$-equivariant morphism. Label by $V$ a component of the fixed point locus $Y^T$ and consider the diagram
\[\begin{CD}
W @>i>> X \\
@VfVV   @VgVV \\
V @>j>> Y \\ 
\end{CD}.\]

Here $W = \amalg_k W_k$ is the (possibly disconnected) set of components of $X^T$ mapping to $V$, and $i$ and $j$ are both inclusions. As torus-fixed loci of smooth schemes (or substacks) $W$ and $V$ will also be smooth.

\begin{lem}[Correspondence of residues]\label{lem: corrres} Let $X$, $Y$ nonsingular schemes or smooth Deligne-Mumford stacks with a $T=(\C^*)^n$-action and $g:X \rightarrow Y$ a proper equivariant morphism. $W=\{W_k\}$ and $V$ are components in the torus-fixed loci of $X$ and $Y$ respectively. Then for $\F$ on $X$, \begin{displaymath} \frac{j^*g_* [\F]}{\lambda_{-1}(N^{\vee}_{V/Y})} =f_* \sum_k \frac{i^* [\F]}{\lambda_{-1}(N_{\{W_k\}/X}^{\vee})}\end{displaymath} \end{lem}

\begin{proof} 
$V$ and $W$ are both nonsingular and so $i$ and $j$ are regular closed embeddings. We have a Gysin map \[j^*: K_0 (Y) \rightarrow K_0(V)\] defined by \[ j^*[\G] = \sum (-1)^i Tor_i^{\OO_Y} (\OO_V, \G),\] and likewise for $i$ (\cite{CFK}).

Equivariant localization implies that equivariant coherent sheaves on $X$ can be written as a sum of sheaves over $X^T$, the torus-fixed part of $X$. Since $f: \{ W_k\} \rightarrow V$ kills all terms in the sum not supported on $\{W_k\}$, we assume that a torus-equivariant coherent sheaf $\F$ on $X$ can be written $[\F] = i_* [\HH]$ for some $\HH \in K^T_0(\{W_k\})$. 

Given $[\F] = i_*[\HH]$, certainly \[j^* g_* [\F] = j^* g_* (i_* [\HH]).\] As $g \circ i = j \circ f$, functoriality of pushforwards implies that the right-hand side then equals $j^* (j_* f_* [\HH])$. The map $j$ is a regular embedding so $j^*j_* (f_*[\HH]) = \lambda_{-1} (N^{\vee}_{V/Y}) \otimes f_*[\HH]$. Thus we have \[j^*g_* [\F] = \lambda_{-1} (N^{\vee}_{V/Y}) f_*[\HH].\] Now apply the fact that $i^* i_*[\HH] = [\HH] \cdot \lambda_{-1} (N^{\vee}_{X^T/X}|_{\{W_k\}})$. \begin{align} j^*g_* [\F] &= \lambda_{-1} (N^{\vee}_{V/Y}) f_*[\HH] \\ &= \lambda_{-1} (N^{\vee}_{V/Y}) f_*\frac{i^* i_* [\HH]}{\lambda_{-1} (N^{\vee}_{X^T/X}|_{\{W_k\}})} \\ &= \lambda_{-1} (N^{\vee}_{V/Y}) f_* \sum_k \frac{i^* [\F]}{\lambda_{-1} (N^{\vee}_{W_k/X})} .\end{align} Thus \begin{equation} \frac{j^*g_* [\F]}{\lambda_{-1} (N^{\vee}_{V/Y})} = f_* \sum_k \frac{i^* [\F]}{\lambda_{-1} (N^{\vee}_{W_k/X})}.\end{equation}

 \end{proof}

\subsection{Functoriality of the J-function for products}\label{sec: func}

In \cite{B}, Bertram proved functoriality of multiplication of J-functions; this proof is extended to K-theory below. 

Let $X_1, \; X_2$ be convex spaces, with $H_2(X_1 \times X_2,\C) = H_2(X_1,\C) \otimes H_2(X_2,\C)$. Thus classes $\beta \in H_2(X_1 \times X_2, \C)$ can be expressed as $(\beta_1, \beta_2) \in H_2(X_1,\C) \otimes H_2(X_2,\C)$. Denote projections by $\pi_1: X_1 \times X_2 \rightarrow X_1$ and $\pi_2: X_1 \times X_2 \rightarrow X_2$. We define the J-functions $J^{X_i,K}(Q_i,q_i)$ as in section \ref{sec: jfunc}.

\begin{prop} For $X_1, \; X_2$ convex spaces, \begin{equation} J^{X_1 \times X_2, K} (Q_1,Q_2,q_1,q_2) = \pi^*_1 J^{X_1,K}(Q_1,q_1) \times \pi^*_2 J^{X_2,K}(Q_2,q_2).\end{equation}\end{prop}
\begin{proof}

 Any map $f:X \rightarrow Y$ induces a map of moduli spaces, \begin{equation} f_{0,m} : \overline{M}_{0,m}(X, \beta) \rightarrow \overline{M}_{0,m}(Y, f_* \beta) \end{equation} as well as maps between the associated graph spaces. If $f$ is equivariant, these maps will be compatible with the induced group action. 

Write $N$ for the normal bundle $N_{M^{X_1 \times X_2}_{(\beta_1, \beta_2)}/G^{X_1 \times X_2}_{(\beta_1, \beta_2)}}$ and $N_i$ for the normal bundles $N_{M^{X_i}_{\beta_i}/G^{X_i}_{\beta_i}}$. (Recall that $M^X_{\beta} = \overline{M}_{0,1}(X, \beta)$.) We must show that \begin{equation}\label{eq: thepoint} ev_* \frac{1}{\lambda_{-1}(N^{\vee})} = \pi^*_1 ev^1_* \frac{1}{\lambda_{-1}(N_1^{\vee})} \otimes \pi^*_2 ev^2_* \frac{1}{\lambda_{-1}(N_2^{\vee})}. \end{equation}

In the following diagram note that when $X_1$, $X_2$ are convex, as assumed, $\Phi$ is a birational morphism.

\begin{equation}\label{eq: firstdiag}\begin{CD}
G^{X_1 \times X_2}_{(\beta_1, \beta_2)} @>{\Phi}>> G^{X_1}_{\beta_1} \times G^{X_2}_{\beta_2} \\
@AAA   @AAA \\
M^{X_1 \times X_2}_{(\beta_1, \beta_2)} @>{\phi}>>  M^{X_1}_{ \beta_1} \times M^{X_2}_{\beta_2} \\
@V{ev}VV  @V{ev^1 \times ev^2}VV \\
X_1 \times X_2 @>=>> X_1 \times X_2
\end{CD}.\end{equation}

 Graph spaces of convex varieties are orbifolds and thus have at worst rational singularities. In K-theory, then,  \begin{equation} \Phi_* [\OO_{G^{X_1\times X_2}_{ (\beta_1, \beta_2)}}] = [\OO_{G^{X_1}_{\beta_1} \times G^{X_2}_{ \beta_2}}]. \end{equation}

Let $\rho_i$ be the projections \[\rho_i: \overline{M}_{0,1}(X_1, \beta_1) \times \overline{M}_{0,1}( X_2, \beta_2)  \rightarrow \overline{M}_{0,1}(X_i, \beta_i),\] Correspondence of residues (Lemma \ref{lem: corrres}) implies \begin{equation}\label{eq: corres} \phi_* \left( \frac{1}{\lambda_{-1}(N^{\vee})} \right) =  \rho_1^* \left( \frac{1}{\lambda_{-1}(N_1^{\vee})} \right) \rho_2^* \left( \frac{1}{\lambda_{-1}(N_2^{\vee})} \right). \end{equation} By the equality at the bottom of the diagram (\ref{eq: firstdiag}), $ev_* = (ev^1 \times ev^2)_* \phi_*$. This combined with (\ref{eq: corres}) gives \begin{align}\label{eq: laststep} ev_*  \left( \frac{1}{\lambda_{-1}(N^{\vee})} \right) &=  (ev^1 \times ev^2)_* \phi_* \left( \frac{1}{\lambda_{-1}(N^{\vee})} \right) \nonumber \\ &= (ev^1 \times ev^2)_*   [\rho_1^* \left( \frac{1}{\lambda_{-1}(N_1^{\vee})} \right) \otimes \rho_2^* \left( \frac{1}{\lambda_{-1}(N_2^{\vee})} \right)].\end{align} Last, observe that for both $i = 1,2$ the following diagram commutes:

\begin{diagram} 
M^{X_1}_{\beta_1}\times M^{X_2}_{\beta_2} & & \\
\dTo^{ev^1 \times ev^2} &\rdTo^{\rho_i} &  \\
X_1 \times X_2 & &M^{X_i}_{\beta_i}\\ 
\dTo^{\pi_i}  &\ldTo_{ev^i} &\\ 
X_i & 
\end{diagram}

Thus  \begin{equation}\label{eq: reallast} (ev^1 \times ev^2)_*   \left[\rho_1^* \left( \frac{1}{\lambda_{-1}(N_1^{\vee})} \right) \otimes \rho_2^* \left( \frac{1}{\lambda_{-1}(N_2^{\vee})} \right) \right]  = \pi^*_1 ev^1_* \frac{1}{\lambda_{-1}(N_1^{\vee})} \otimes  \pi^*_2 ev^2_* \frac{1}{\lambda_{-1}(N_2^{\vee})} . \end{equation}
Multiplicativity of the K-theoretic J-function follows from equations (\ref{eq: laststep}) and (\ref{eq: reallast}). \end{proof}

\subsection{Pushforward}\label{sec: pushf}

A lemma on pushforwards, analogous to the many other Gysin map lemmas for cohomology (\cite{BCK1,BCK2, FulPrag, Prag, Br, AC}), is needed for computation at the end of section (\ref{sec: quot}). First recall some properties of homogeneous spaces $G/P$, for $G$ a Lie group and $P$ a parabolic subgroup in $G$. Under the action of the maximal torus $T \subset P \subset G$, $G/P$ has finitely many isolated fixed points indexed by permutations $w$ in the Weyl group $W/W_P$ of $G/P$. Label these fixed points $q_w$. Inclusion $i_w: q_w \hookrightarrow G/P$ induces pullback $i^*_w: K_0^TT(G/P) \rightarrow K_0^T(pt)$. The tangent space to $G/P$ at a torus-fixed point $q_w$ is isomorphic to $\mathfrak{g} / \mathfrak{p}$, and the torus acts with weights $w(\alpha)$, where $\alpha \in \Delta^+$ are the positive roots. Moreover, the pullback $i_w^* \F$ for a T-equivariant vector bundle $\F$ is equivalent to $w[\F]$: we write $[\F]$ as a polynomial in the $\LL_i$ (``Chern line bundles'') and let $w$ act on the indices.

We eventually wish to push forward from a flag bundle to a Grassmannian. Consider then parabolic subgroups $P$ and $P'$, $T \subset P' \subset P \subset G$. The pushforward is $\pi: G/P \rightarrow G/P'$. The Weyl group $W/W_P$ acts on $K_0^T(G/P)$ and the pullback map $\pi^*$ identifies $K_0^T(G/P')$ with $K_0^T(G/P)^{W/W_P'}$.

We have a diagram
\[\begin{CD}
\amalg_{w} \{q_{w}\} @>{i_{w}}>> G/P \\
@V{f}VV   @VV{\pi}V \\
\amalg_{w'}\{ q_{w'}\} @>{j_{w'}}>> G/P'
\end{CD}\] where $j_{w'}: \{p_{w'}\} \hookrightarrow G/P'$ and $i_{ w} : \{q_{w}\} \hookrightarrow G/P$ are the inclusions.

\begin{lem}\label{thm: pushf} Let $X = G/P$ and $Y = G/P'$, for $G$ a compact connected Lie group and $T \subset P' \subset P \subset G$, $P,P'$ parabolic subgroups. Consider $[\F] \in K^T_0(X)$. In $K_0^T(Y)$, \begin{equation} \pi_* [\F] = \sum_{w \in W/W_P} (-1)^{\epsilon (w)} \frac{w([\F])}{\lambda^{\C^*}_{-1} (T^{\vee}_{\pi})}.\end{equation} The right-hand side of the equality is invariant under $W/W_{P'}$, hence may be viewed as a class on $Y$.\end{lem}

\begin{proof} Start with the correspondence of residues. From lemma (\ref{lem: corrres}), we know that  \begin{displaymath}  \frac{j_{w'}^*\pi_* [\F]}{\lambda_{-1}(N^{\vee}_{q_{w'}/Y})} =f_* \sum_w \frac{i_w^* [\F]}{\lambda_{-1}(N_{\{q_w\}/X}^{\vee})}.\end{displaymath} Apply $j_* = \sum_{w'} j_{w'*}$ to both sides: then the left hand side reduces to $\pi_* [\F]$ by the ABBV theorem. On the right-hand side, $i^*_w [\F] = w[\F]$ as noted above, and $\lambda_{-1}(N_{q_w/X}^{\vee})$ evaluated at a given $w$ is $(-1)^{\epsilon (w)} \lambda_{-1} (N^{\vee}_{id/X})$, where $N^{\vee}_{id/X}$ is the normal/tangent bundle with torus action given by positive weights.

Thus\begin{align}j_*f_* \sum_{w \in W/W_P} \frac{i_w^* [\F]}{\lambda_{-1}(N_{\{q_w\}/X}^{\vee})} &= j_*f_* \sum_{w \in W/W_P} \frac{w [\F]}{(-1)^{\epsilon(w)} \lambda_{-1}(N_{id/X}^{\vee})} \\ &= j_* \sum_{w \in W/W_P} \frac{w [\F]}{(-1)^{\epsilon(w)} \lambda_{-1}(N_{id/X}^{\vee})} \\&=   \sum_{w \in W/W_P} \frac{w [\F]}{(-1)^{\epsilon(w)} \lambda_{-1}(N_{id/X}^{\vee})}\cdot \lambda_{-1}(N^{\vee}_{q_{w'}/Y})\\ &= \sum_{w \in W/W_P} \frac{w [\F]}{(-1)^{\epsilon(w)} \lambda_{-1}(T_{\pi}^{\vee})}.\end{align} Exactness of pushforward to a point justifies the second equality, and the third equality comes from $j_*j^* [\G] = [\G] \otimes \lambda_{-1}(N^{\vee}_{q_{w'}/Y})$ for $[\G] \in K^T_0(Y)$. We rewrite to acknowledge that $T_{\pi}= T_X - \pi^*T_Y$. Summing over $w \in W/W_P$ will automatically sum over all $w' \in W/W_{P'}$, and the sum will be $W/W_{P'}$-invariant.

\end{proof}

For later use, we compute this explicitly in terms of line bundles for the pushforward from a type A flag manifold $\Fl(m_1, \ldots, m_{\ell}, m_{\ell+1}=r)$ to a type A Grassmannian $Gr(r,n)$ with tautological bundle $S$. Compute $[T_{\pi}] = [T\Fl]-\pi^*[TGr]$. The exact sequences for these tangent bundles are, again, \begin{align*} 0 \rightarrow S^{\vee} \otimes S \rightarrow S^{\vee} &\otimes \C^n \rightarrow TGr \rightarrow 0 \\ 0 \rightarrow K \rightarrow \pi^* S^{\vee} &\otimes \C^n \rightarrow TFl \rightarrow 0. \end{align*} $K$ has a filtration, \begin{equation*} 0 = K_0 \subset K_1 \subset \cdots \subset K_{\ell-1} \subset K_{\ell}=K ,\end{equation*} and each quotient can be written in terms of the tautological bundles $S_{m_i}$ of the flag manifold:  $K_i/K_{i-1} = (S_{m_i}/S_{m_{i-1}})^{\vee} \otimes S_{m_i}$. Each of these filters further as \[  (S_{m_i}/S_{m_{i-1}})^{\vee} \otimes S_{m_1} \subset  (S_{m_i}/S_{m_{i-1}})^{\vee} \otimes S_{m_2 } \subset \cdots \subset (S_{m_i}/S_{m_{i-1}})^{\vee} \otimes S_{m_i}.\]  Thus \begin{align*} [T_{\pi}] &= [T\Fl] -\pi^* [TGr] \\ &= \pi^*[S^{\vee} \otimes \C^n] - [\pi^* S^{\vee} \otimes \C^n] +\pi^*[ S^{\vee} \otimes S] - \sum_{i \leq j} [(S_{m_j}/S_{m_{j-1}})^{\vee} \otimes (S_{m_i}/S_{m_{i-1}})] .\end{align*} Using the splitting principle to write the tautological bundles of the flag in terms of $\sum_{s=1}^{m_j-m_{j-1}} \LL_{m_{j-1}+s} = S^{\vee}_{m_j}$, \begin{align}\label{eq:splitex} \lambda_{-1}^{\C^*} (T^{\vee}_{\pi}) &= \frac{\lambda_{-1}^{\C^*}( \pi^*[S \otimes S^{\vee}]}{\lambda_{-1}^{\C^*} ((\sum_{i \leq j} [(S_{m_j}/S_{m_{j-1}})^{\vee} \otimes (S_{m_i}/S_{m_{i-1}})])^{\vee})} \\ &= \frac{\prod_{i,j} (1- \LL_i^{\vee} \otimes \LL_j)}{\prod_{i \leq j} \prod_{s,t} (1- \LL_{m_{j-1}+s}^{\vee} \otimes \LL_{m_{i-1}+t}) } \\ &= \prod_{i > j} \prod_{s,t} (1- \LL^{\vee}_{m_{j-1}+s} \otimes \LL_{m_{i-1}+t}).\end{align} This is a class in $K_0^T(Gr(r,n))$ because it is invariant under the action of $S_r \times S_{n-r}$.

\section{The J-function of the Grassmannian}\label{sec: quot}

The K-theoretic J-function of the Grassmannian is interesting not only for its own sake, but because it will be used in every other J-function computation hereafter.

\begin{thm} The K-theoretic J-function for the Grassmannian $Gr(r,n)$ is 
\begin{displaymath} J^{Gr(r,n),K}(Q,q) = \sum_d Q^d J^{Gr(r,n),K}_d(q) \end{displaymath} where
\begin{displaymath} J^{Gr(r,n),K}_d(q) =  \sum_{d_1+\cdots + d_r = d} (-1)^{(r-1)d} \left( \frac{\prod_{1 \leq j <i \leq r}  (1-\LL^{\vee}_i \otimes \LL_j q^{d_i-d_j})}{\prod_{1 \leq j < i \leq r} (1- \LL^{\vee}_i \otimes \LL_j) \prod_{i=1}^{r} \prod_{\ell = 1}^{d_{i}} (1 - \LL^{\vee}_i q^{\ell})^n} \right).\end{displaymath}\end{thm}

\begin{proof} The proof that follows is quite analogous to the proof by localization for the cohomological version in \cite{BCK1}. The reader is referred to \cite{BCK1} for the proofs of the statements of the lemmas (\ref{thm: 11inBCK}, \ref{thm: 12inBCK}) below on the torus-fixed loci of Quot schemes.

To understand torus-fixed points of $\Quot_{\PP^1,d}(\C^n, n-r)$, $\Quot_{\PP^1,d} (\C^r,0)$ must be understood. The points of $\Quot_{\PP^1,d} (\C^r,0)$ are degree $d$ torsion quotients of $\C^r \otimes \OO_{\PP^1}$. The top exterior power of the kernel vector bundle gives a map to $\PP^d$: \begin{align}  \wedge^r: \Quot_{\PP^1,d}(\C^r,0) &\rightarrow \Quot_{\PP^1,d}(\C,0) =\PP^d \\ \wedge^r(K \subset \C^r \otimes \OO_{\PP^1}) &\rightarrow ( \wedge^r K) \subset \C \otimes \OO_{\PP^1}. \end{align} This kernel $K$ is a locally free sheaf on $\PP^1$ and splits as $K \cong \OO (d_1) \oplus \OO (d_2) \oplus \ldots \oplus \OO (d_r)$, $\sum d_i = d$. Require that $0 \leq d_1 \leq d_2 \leq \cdots \leq d_r$ to make this splitting unique. Then:

\begin{lem}[1.1 in \cite{BCK1}]\label{thm: 11inBCK} For each splitting type $\{d_i\}$ as above, let $m_1 < m_2 < \cdots m_k$ denote the jumping indices (i.e., $0 \leq d_1 = \cdots d_{m_1} < d_{m_1+1} = \cdots =d_{m_2} < \cdots)$. Then there is an embedding of the flag manifold: \begin{displaymath} i_{\{d_i\}} :\Fl (m_1, m_2, \ldots, m_k; r) \hookrightarrow \Quot_{\PP^1, d} (\C^r,0)\end{displaymath} with the property that each fixed point of the $\C^*$-action with $\supp (\mathcal{Q}) = \{ 0 \}$ and with the kernel splitting type $\{ d_i \}$ corresponds to a point of the image of $i_{\{d_i \}}$. \end{lem}

Recall from our pushforward example the universal flag on $\Fl (m_1, m_2, \ldots, m_k; r)$: \begin{displaymath} 0 \subseteq S_{m_1} \subset S_{m_2} \subset \cdots S_{m_k} \subset S_{m_{k+1}} = \C^r \otimes \OO_{\Fl}. \end{displaymath} Denote the projection from $\PP^1 \times \Fl$ to $\Fl$ by $\pi$. Construct a sheaf $\K$ over $\PP^1 \times \Fl$ so that $\K:= S^{(k)}_{m_{k+1}}$ where \begin{displaymath} 0 \subset \pi^* S_{m_1} \subset S^{(1)}_{m_2} \subset \cdots \subset S^{i-1}_{m_i} \subset \cdots \subset S^{(k)}_{m_{k+1}} \subset \C^r \otimes \OO_{\PP^1 \times \Fl}\end{displaymath} and \begin{displaymath}S^{(i-1)}_{m_i}/S^{(i-2)}_{m_{i-1}} \cong \pi^* (S_{m_i}/S_{m_{i-1}})(-d_{m_i}z). \end{displaymath}

\begin{lem}[1.2 in \cite{BCK1}]\label{thm: 12inBCK} There is a natural $\C^*$-equivariant embedding \begin{displaymath} j: \Quot_{\PP^1,d}(S,0) \hookrightarrow \Quot_{\PP^1,d} (\C^n, n-r)\end{displaymath} such that all the fixed points of the $\C^*$-action on $\Quot_{\PP^1,d} (\C^n, n-r)$ are contained in the image. The fixed points of $\Quot_{\PP^1,d}(\C^n, n-r)$ that also satisfy $\supp( \mathrm{tor} (\mathcal{Q})) = \{ 0 \}$ (the support of the torsion part of $\mathcal{Q}$) are precisely the images of flag manifolds \begin{displaymath} i_{\{d_i \}} : \Fl (m_1, m_2, \ldots, m_k, r;n) = \Fl (m_1, m_2, \ldots, m_k; S) \hookrightarrow \Quot_{\PP^1, d}(S,0)\end{displaymath} embedded by the relative version of Lemma 1.1.\end{lem}

Use the following diagram to justify a K-theoretic version of lemma 1.3 from \cite{BCK1}, which shifts the calculation of the residue to the Quot scheme and allows us to exploit the flag variety structure of the Quot scheme's torus-fixed loci.
\begin{diagram} 
\Quot_{\PP^1,d}(\C^n,n-r) &\rTo_v &\PP_d^{\binom{n}{r}-1} &\lTo_u &G^{Gr(r,n)}_d) \\
\uInto & &\uInto & &\uInto \\
\amalg_{\{d_i\}} i_{\{d_i\}} (\Fl) &\rInto_{p} &\PP^{\binom{n}{r}-1} &\lInto_{q} &M^{Gr(r,n)}_d\\ 
&\rdTo_{\rho} &\uInto &\ldTo_{ev}\\ 
& &Gr(r,n) 
\end{diagram}

\begin{lem} \begin{equation} J_d^{Gr,K} (q) = \sum_{\{d_i\}} \rho_* \left( \frac{[\OO_{i_{\{d_i\}}\Fl}]}{\lambda_{-1}^{\C^*} (T^{\vee}_{i_{\{d_i\}} \Fl})} \right), \end{equation} where $\rho: \Fl (m_1, \ldots, m_k,r,n) \rightarrow Gr(r,n)$. \end{lem} \begin{proof}  The second row of the diagram consists of $\C^*$-fixed loci of the first row. Use correspondence of residues twice and the fact that $u_*[\OO_{G^{Gr(r,n)}_d}] = v_*[\OO_{\Quot_{\PP^1,d}(\C^n, n-r)}]$ to see that \[q_* \left( \frac{[\OO_{M^{Gr(r,n)}_d}]}{\lambda_{-1}^{\C^*} (N^{\vee}_{M^{Gr(r,n)}_d / G^{Gr(r,n)}_d})} \right) = \sum_{\{d_i\}} p_* \left( \frac{[\OO_{i_{\{d_i \}}\Fl_{\{d_i\}}}]}{\lambda_{-1}^{\C^*} (T^{\vee}_{i_{\{d_i\}} \Fl})} \right).\] (See \cite{SS} for a proof that the image of $u$ and $v$, $Gr(r,n)_d$, has rational singularities.) However, this equality is on $\PP^{\binom{n}{r}-1}$ rather than $Gr(r,n)$. Add an additional action by $T = (\C^*)^n$ with weights $\lambda_1, \ldots, \lambda_n$ to obtain an injective map induced by the Pl\"{u}cker embedding \[ K_T^* (Gr(r,n)) \rightarrow K_T^*(\PP^{\binom{n}{r}-1}). \]

This allows us to view the following classes in $K_T(\PP^{\binom{n}{r}-1}$ as classes in $K_T(Gr(r,n))$: \[q_{T*} \left( \frac{[\OO_{M^{Gr(r,n)}_d}]}{\lambda_{-1}^{\C^* \times T} (N^{\vee}_{M^{Gr(r,n)}_d / G^{Gr(r,n)}_d})} \right) = \sum_{\{d_i\}} p_{T*} \left( \frac{[\OO_{\Fl_{\{d_i\}}}]}{\lambda_{-1}^{\C^* \times T} (T^{\vee}_{i_{\{d_i\}} \Fl})} \right) \in K_T^* (Gr(r,n)).\] 
In particular, their restrictions are also in $K_T(Gr(r,n))$:  \[ev_{T*} \left( \frac{[\OO_{M^{Gr(r,n)}_d}]}{\lambda_{-1}^{\C^* \times T} (N^{\vee}_{M^{Gr(r,n)}_d / G^{Gr(r,n)}_d})} \right) = \sum_{\{d_i\}} \rho_{T*} \left( \frac{[\OO_{i_{\{d_i \}}\Fl_{\{d_i\}}}]}{\lambda_{-1}^{\C^* \times T} (T^{\vee}_{i_{\{d_i\}} \Fl})} \right) \in K_T^* (Gr(r,n)).\] As $\lambda_i \rightarrow 0$ the $\C^* \times T$-equivariant terms go to the $\C^*$-equivariant terms:   \[ev_{*} \left( \frac{[\OO_{M^{Gr(r,n)}_d}]}{\lambda_{-1}^{\C^*} (N^{\vee}_{M^{Gr(r,n)}_d / G^{Gr(r,n)}_d})} \right) = \sum_{\{d_i\}} \rho_{*} \left( \frac{[\OO_{i_{\{d_i \}}\Fl_{\{d_i\}}}]}{\lambda_{-1}^{\C^*} (T^{\vee}_{i_{\{d_i\}} \Fl})} \right) \in K^* (Gr(r,n)).\]   \end{proof}

To find the J-function for the Grassmannian, then, we need to find $\lambda_{-1}^{\C^*}(T^{\vee}_{i_{\{d_i\}} \Fl})$ and push it forward using localization. We know that \begin{align} \lambda_{-1}^{\C^*}(T_{i_{\{d_i\}} \Fl}^{\vee}) &= \lambda_{-1}^{\C^*}(T\Quot|_{i_{\{d_i\}} Fl}^{\vee} / TFl^{\vee}) \\ &= \frac{\lambda_{-1}^{\C^*}((\pi_* \K^{\vee} / \pi^* S^{\vee})^{\vee})^n \lambda_{-1}^{\C^*}((R^1\pi_* \K^{\vee} \otimes \K)^{\vee})}{\lambda_{-1}^{\C^*}((\pi_*(\K^{\vee} \otimes \K)/K)^{\vee})}. \end{align} The second equality comes from the following exact sequences: \begin{equation} 0 \rightarrow \pi_* (\K^{\vee} \otimes \K) \rightarrow \pi_* \K^{\vee} \otimes \C^n \rightarrow T\Quot \rightarrow R^1 \pi_* \K^{\vee} \otimes \K \rightarrow 0\end{equation} \begin{equation} 0 \rightarrow K \rightarrow \rho^* S^{\vee} \otimes \C^n \rightarrow TFl(m_1, \ldots, m_k, S) \rightarrow 0. \end{equation}

The filtration of $K$ worked out in the pushforward example (see section (\ref{sec: pushf})) gives \begin{equation} [TFl] = n [\rho^* S^{\vee}] - \sum_{i \leq j} [(S_{m_j}/S_{m_{j-1}})^{\vee} \otimes (S_{m_i}/S_{m_{i-1}})]. \end{equation} Once again, the splitting principle gives this in terms of line bundles: \[\sum_{s=1}^{m_i-m_{i-1}} \LL_{m_{i-1}+s} = (S_{m_i}/S_{m_{i-1}})^{\vee}\] in the Grothendieck group of $Fl$. Thus \begin{equation} [TFl] = n \sum_i \sum_s \LL_{m_{i-1}+s} - \sum_{i \leq j} \sum_{s,t} (\LL_{m_{j-1}+s} \otimes \LL_{m_{i-1}+t}^{\vee}) .\end{equation} Taking the equivariant Euler class of the dual of the previous, \begin{equation} \lambda_{-1}^{\C^*} (TFl^{\vee}) = \frac{\prod_{i=1}^{\ell} \prod_{s=1}^{m_i-m_{i-1}} (1 - \LL^{\vee}_{m_{i-1}+s})^n}{\prod_{i \leq j} \prod_{s,t} (1 - \LL^{\vee}_{m_{j-1}+s} \otimes \LL_{m_{i-1}+t})}. \end{equation}

Moving to $T\Quot$, note that we restrict to $\PP^1 \times Fl$ from $\PP^1 \times \Quot$, and then project by $\pi: \PP^1 \times Fl \rightarrow Fl$. We will use the fact that \begin{equation} [\K^{\vee} \otimes \K ] = \sum_{i,j =1}^{k+1} [\pi^* ((S_{m_i}/S_{m_{i-1}})^{\vee} \otimes (S_{m_j}/S_{m_{j-1}})) (d_{m_i}-d_{m_j})]\end{equation} It is easiest to consider the complex $R^{\bullet} \pi_* (\K^{\vee} \otimes \K)$, because many cancellations will occur between the two terms.

Compute by looking at the complex $R^{\bullet} \pi_* (\K^{\vee} \otimes \K)$ fiberwise. \begin{multline}[R^{\bullet} \pi_* (\K^{\vee} \otimes \K)] \cong [H^0(\PP^1, \oplus_{i,j = 1}^{r+1} \pi^* ((S_{m_i}/S_{m_{i-1}})^{\vee}) \otimes (S_{m_j}/S_{m_{j-1}})(d_{m_i} - d_{m_j}))] \\ -[H^1(\PP^1, \oplus_{i,j = 1}^{r+1} \pi^* ((S_{m_i}/S_{m_{i-1}})^{\vee}) \otimes (S_{m_j}/S_{m_{j-1}})(d_{m_i} - d_{m_j}))] \\ = [H^0(\PP^1, \oplus_{1 \leq j \leq i \leq r+1} \pi^* ((S_{m_i}/S_{m_{i-1}})^{\vee}) \otimes (S_{m_j}/S_{m_{j-1}})(d_{m_i} - d_{m_j}))] \\ -[H^1(\PP^1, \oplus_{1 \leq i < j \leq r+1} \pi^* ((S_{m_i}/S_{m_{i-1}})^{\vee}) \otimes (S_{m_j}/S_{m_{j-1}})(d_{m_i} - d_{m_j}))] \\  \end{multline} In the last line we drop the portions of the fiber bundle that don't contribute because of the twist by $\OO(d_{m_i} - d_{m_j})$.

Use Serre duality on the $H^1$ term: \begin{multline}  [H^1(\PP^1, \oplus_{1 \leq i < j \leq r+1} \pi^* ((S_{m_i}/S_{m_{i-1}})^{\vee}) \otimes (S_{m_j}/S_{m_{j-1}})(d_{m_i} - d_{m_j}))] \\  \cong [H^0(\PP^1, \oplus_{1 \leq i < j \leq r+1} \pi^* ((S_{m_i}/S_{m_{i-1}})) \otimes (S_{m_j}/S_{m_{j-1}})^{\vee}(-d_{m_i} + d_{m_j}-2))]^{\vee}.\end{multline} 

In order to get equivariant Euler characteristic, rewrite using splitting principle as sum of line bundles: \begin{multline}[R^{\bullet} \pi_* (\K^{\vee} \otimes \K)] =[ \oplus_{1 \leq j \leq i \leq r+1} (\oplus_s \LL_{m_{i-1}+s}) \otimes (\oplus_t \LL^{\vee}_{m_{j-1}+t})(d_{m_i}-d_{m_j})] \\ - [ \oplus_{1 \leq i <j \leq r+1} (\oplus_s \LL^{\vee}_{m_{i-1}+s}) \otimes (\oplus_t \LL_{m_{j-1}+t})(-d_{m_i}+d_{m_j}-2)].  \end{multline} Since the dual appears on the line bundle with smaller index on both sides, reindex so that $\LL^{\vee}$ has index $j$, $j <i$.\begin{multline}[R^{\bullet} \pi_* (\K^{\vee} \otimes \K)] =[ \oplus_{1 \leq j \leq i \leq r+1} (\oplus_s \LL_{m_{i-1}+s}) \otimes (\oplus_t \LL^{\vee}_{m_{j-1}+t})(d_{m_i}-d_{m_j})] \\ - [ \oplus_{1 \leq j <i \leq r+1} (\oplus_s \LL_{m_{i-1}+s}) \otimes (\oplus_t \LL^{\vee}_{m_{j-1}+t})(-d_{m_j}+d_{m_i}-2)].  \end{multline}
When we take equivariant Euler classes, we decompose according to the character $q^k$ of the $S^1$-action on $\PP^1$: $\lambda_{-1}^{\C^*} (\oplus_i \LL_i \otimes \OO(d)) = \prod_i \prod_{k = 0}^d (1- \LL_i q^k)$. Looking at the $S^1$-action on $H^0(\PP^1, F (d))$ and $H^1(\PP^1, F(d))^{\vee}$ tells us that all sections ``in the middle'' cancel, and we are left with \begin{multline}\lambda_{-1}^{\C^*} ((R^0 \pi_* (\K^{\vee} \otimes \K) \ominus R^1 \pi_* (\K^{\vee} \otimes \K))^{\vee}) \\ = \prod_{i >j} \prod_{s,t}(-1)^{r_i r_j (d_{ij}-1)} (1- \LL^{\vee}_{m_{i-1}+s} \otimes \LL_{m_{j-1}+t} q^{d_{ij}})(1- \LL^{\vee}_{m_{i-1}+s} \otimes \LL_{m_{j-1}+t}),\end{multline} where $r_i = m_i -m_{i-1}$ and $d_{ij} =d_i -d_j$. In addition, there is a term with $i=j$ that cancels with an equal contribution from $K$ below.

The rest of the calculation proceeds more easily:

\begin{align} \lambda_{-1}^{\C^*} (K^{\vee}) &= \lambda_{-1}^{\C^*} ([\oplus_{1 \leq i \leq j \leq r+1} (S_{m_j}/S_{m_{j-1}})^{\vee} \otimes (S_{m_i}/S_{m_{i-1}})]^{\vee}) \\ &= \prod_{1 \leq i \leq j \leq r+1} \prod_{s,t} (1- \LL^{\vee}_{m_{j-1}+s} \otimes \LL_{m_{i-1}+t}) \end{align} Notice that within is the piece with $i = j$ that cancels exactly the same contribution from $[R^{\bullet} \pi_* (\K^{\vee} \otimes \K)]$.

Finally, \begin{align} \lambda_{-1}^{\C^*} ((\pi_* \K )^n) &= \lambda_{-1}^{\C^*}( ( \pi_* (\oplus_i \pi^*(S_{m_i}/S_{m_{i-1}})^{\vee}(d_{m_i})))^n)^{\vee} \\ &= \prod_{i=1}^{r+1} \prod_{s=1}^{r_i} \prod_{k=0}^{d_{m_i}} (1- \LL^{\vee}_{m_{i-1}+s} q^k)^n  \end{align} while \begin{equation} \lambda_{-1}^{\C^*} ((\pi_* S)^n) = \prod_{i=1}^{r+1} \prod_{s=1}^{r_i} (1- \LL^{\vee}_{m_{i-1}+s})^n . \end{equation} Thus \begin{equation} \lambda_{-1}^{\C^*} (\pi_* \K/\pi^{\vee} S )^n  = \prod_{i=1}^{r+1} \prod_{s=1}^{r_i} \prod_{k=1}^{d_{m_i}} (1- \LL^{\vee}_{m_{i-1}+s} q^k)^n .\end{equation}

Combine all of the above, with a little bit of re-indexing here and there: 

\begin{equation}\label{eq: FinQYd} \lambda_{-1}^{\C^*} (T^{\vee}_{i_{\{d_i\}}\Fl}) = \frac{\prod_{i=1}^{r+1} \prod_s \prod_{\ell = 1}^{d_{m_i}} (1 - \LL^{\vee}_{m_{i-1}+s} q^{\ell})^n}{\prod_{i >j} (-1)^{r_ir_j(d_{ij}-1)} \prod_{s,t} (1-\LL^{\vee}_{m_{i-1}+s} \otimes \LL_{m_{j-1}+t} q^{d_{ij}})}.\end{equation}

This is what we now need to push forward, as \begin{equation} J^{Gr(r,n),K}_d( \hbar) = \sum_{\{d_i\}} \rho_* \left( \frac{\prod_{i >j} (-1)^{r_ir_j(d_{ij}-1)} \prod_{s,t} (1-\LL^{\vee}_{m_{i-1}+s} \otimes \LL_{m_{j-1}+t} q^{d_{ij}})}{\prod_{i=1}^{r+1} \prod_s \prod_{\ell = 1}^{d_{m_i}} (1 - \LL^{\vee}_{m_{i-1}+s} q^{\ell})^n} \right).\end{equation} Using the pushforward lemma (\ref{thm: pushf}), we know that \begin{equation}\sum_{\{d_i\}} \rho_* P = \sum_{\{d_i\}} \sum_w w \left[ \frac{P}{\lambda_{-1}^{\C^*} (T^{\vee}_{\pi})} \right] \end{equation} for $w \in S_r/(S_{r_1} \times \cdots \times S_{r_{k+1}})$.

Finally, use the reasoning from \cite{BCK1}: given any partition $d_i$ of $d$ with $k+1$ distinct $d_i$s and $r_i$ denoting multiplicity, there is a unique $w \in S_r / (S_{r_1} \times \cdots \times S_{r_{k+1}})$ such that $w^{-1}$ arranges $(d_1, \ldots, d_r)$ in nondecreasing order $d_1 \leq \cdots \leq d_r$. Use this to replace the double sum:
 \begin{align} \begin{split}J^{Gr(r,n),K}_d( q) &= \sum_{\{d_i\}} \sum_w\left( \frac{\prod_{j <i}\prod_{s,t} (1-\LL^{\vee}_{m_{i-1}+s} \otimes \LL_{m_{j-1}+t} q^{d_{ij}})}{\prod_{i=1}^{r+1} \prod_s \prod_{\ell = 1}^{d_{m_i}} (1 - \LL^{\vee}_{m_{i-1}+s} q^{\ell})^n} \right. \\ &\quad\cdot \left. \frac{1}{\prod_{j <i} (-1)^{r_ir_j(d_{ij}-1)} \prod_{s,t} (1-\LL^{\vee}_{m_{j-1}+s} \otimes \LL_{m_{i-1}+t})} \right) \end{split} \\ &=  \sum_{d} (-1)^{(r-1)d} \left( \frac{\prod_{1 \leq j <i \leq r}  (1-\LL^{\vee}_i \otimes \LL_j q^{d_{ij}})}{\prod_{1 \leq j < i \leq r} (1- \LL^{\vee}_i \otimes \LL_j) \prod_{i=1}^{r} \prod_{\ell = 1}^{d_{i}} (1 - \LL^{\vee}_i q^{\ell})^n} \right).\end{align} This process will be used in all further computations of K-theoretic J-functions.
\end{proof}

\subsection{Application: the abelian-nonabelian correspondence}
One corollary of this result is the analogue of a result in \cite{BCK1}, a proof of one case of the abelian-nonabelian correspondence. The abelian-nonabelian correspondence is a broad conjectural relationship between the Gromov-Witten theories and Frobenius structures of GIT quotients of a space $V$ by a group $G$ and its maximal torus $T \subset G$. In this case, we consider $V = \Hom (\C^r, \C^n)$ and $G = GL_k$. Then $V//G = Gr(r, n)$ and $V//T = (\PP^{n-1})^r$. In \cite{BCK1} it is shown that \begin{equation} J^{Gr(r,n)} = e^{-\sigma_1(r-1)\pi \sqrt{-1}/\hbar} \frac{\mathcal{D}_{\Delta} J^{(\PP^{n-1})^r}}{\Delta}|_{t_i = t+(r-1)\pi \sqrt{-1}}, \end{equation} where \begin{equation*} \Delta = \prod_{i < j} (x_i - x_j) \quad \mathrm{and} \quad \mathcal{D}_{\Delta} = \prod_{i<j} \left( \hbar \frac{\partial}{\partial t_i } - \hbar \frac{\partial}{\partial t_j} \right),\end{equation*} and $t_i$ a basis for $H_2(Gr(r,n), \C)$, $x_i$ Chern roots of the dual of the tautological bundle $S^{\vee}$, and $\sigma_1$ the basis for $H^2(Gr(r,n),\C)$.

We must make modifications for the K-theoretic version, and take inspiration from \cite{GL}. Use the operator \begin{equation} q^{\partial /\partial t_i}: t_j \mapsto t_j + \delta_{ij} \ln q \end{equation} to define \begin{equation}\mathcal{D}_{\Delta} := \prod_{i>j} (q^{\partial /\partial t_i} -q^{\partial /\partial t_j}).\end{equation} In addition, we define \begin{equation} \Delta := \prod_{i>j} (\LL^{\vee}_i - \LL^{\vee}_j) = \prod_{i>j} (1- \LL_i^{\vee} \otimes \LL_j).\end{equation} Notice the switch from $i<j$ to $i>j$: this is because in K-theory our notion of ``positivity'' is different. (See work of \cite{GK, GR, AGM} for more on positivity in K-theory.) Last, but not least, we must use a modified version of the J-function: $(\LL^{\vee})^{\ln Q/ \ln q} J^{K}$, where $\LL^z = \LL_1^{z_1} \cdots \LL_r^{z_r}$.

\begin{cor}  \begin{equation} \frac{\mathcal{D}_{\Delta}(J^{(\PP^{n-1})^r,K})}{\Delta}|_{t_i =t+(r-1) \pi \sqrt{-1}} = (\LL^{\vee})^{\ln Q / \ln q} J^{Gr(r,n),K} \end{equation}\end{cor}

\begin{proof} Recall that the K-theoretic J-function for the Grassmannian $Gr(r,n)$ is 
\begin{displaymath} J^{Gr(r,n),K}(Q,q) = \sum_d J^{Gr(r,n),K}_d(q) Q^d\end{displaymath} where
\begin{displaymath} J^{Gr(r,n),K}_d(q) =  \sum_{d_1+\cdots + d_r = d} (-1)^{(r-1)d} \left( \frac{\prod_{1 \leq j <i \leq r}  (1-\LL^{\vee}_i \otimes \LL_j q^{d_i-d_j})}{\prod_{1 \leq j < i \leq r} (1- \LL^{\vee}_i \otimes \LL_j) \prod_{i=1}^{r} \prod_{\ell = 1}^{d_{i}} (1 - \LL^{\vee}_i q^{\ell})^n} \right).\end{displaymath}

Applying $q^{\partial / \partial t_i}$ to $(\LL^{\vee})^{\ln Q/ \ln q} J^{(\PP^{n-1})^r,K}$, we get \begin{equation} q^{\partial / \partial t_i} ((\LL^{\vee})^{\ln Q/ \ln q} J^{(\PP^{n-1})^r,K}) = \sum_{d_1 + \cdots +d_k =d} \LL_i^{\vee} q^{d_i} ((\LL^{\vee})^{\ln Q/ \ln q} J_d^{(\PP^{n-1})^r,K}). \end{equation} Thus the application of $\mathcal{D}_{\Delta}/ \Delta$ to the same results in \begin{align} \frac{\mathcal{D}_{\Delta}}{\Delta} ((\LL^{\vee})^{\ln Q/ \ln q} J^{(\PP^{n-1})^r,K}) &= \sum_d \prod_{i>j} \frac{(\LL_i^{\vee}q^{d_i} - \LL_j^{\vee}q^{d_j})}{(\LL^{\vee}_i - \LL^{\vee}_j)}((\LL^{\vee})^{\ln Q/ \ln q} J_d^{(\PP^{n-1})^r,K}) \\ &=  \prod_{i>j} \frac{(1-\LL_i^{\vee}\otimes \LL_j q^{d_i-d_j})}{(1-\LL_i^{\vee} \otimes \LL_j)} ((\LL^{\vee})^{\ln Q/ \ln q} J_d^{(\PP^{n-1})^r,K}).\end{align} When $t_i$ is specialized to $t+(r-1) \pi \sqrt{-1}$, $Q_i = e^{t_i}$ goes to $e^{t+ (r-1) \pi \sqrt{-1}} = (-1)^{(r-1)}Q$, and so  

\begin{multline} \frac{\mathcal{D}_{\Delta}(J^{(\PP^{n-1})^r,K})}{\Delta}|_{t_i =t+(r-1) \pi \sqrt{-1}} = \\ (\LL^{\vee})^{\ln Q/ \ln q} \sum_d (-1)^{(r-1)d} Q^d \prod_{i>j} \frac{(1-\LL_i^{\vee}\otimes \LL_j q^{d_i-d_j})}{(1-\LL_i^{\vee} \otimes \LL_j) \prod_{i=1}^r \prod^d_{\ell=1} (1-q^{\ell} \LL_i^{\vee})^{n} },\end{multline} which is $(\LL^{\vee})^{\ln Q / \ln q} J^{Gr(r,n),K}$.

\end{proof}

\section{The J-function of flag varieties of type A}
\subsection{A simple lemma}\label{sec: simlem}

To calculate the K-theoretic J-function of the flag manifold, more machinery is needed. We will look at the flag manifold as a subvariety of a product of Grassmannians, all embedded in projective space: $Fl(m_1, \ldots, m_{\ell}=k,n) \subset \prod_{i =1}^{\ell} Gr(m_i,n) \subset \PP^N$. A more general result can be proved for homogeneous varieties $X \subset Y \subset \PP^N$. In all the following $Y$ will be a Grassmannian or a product of Grassmannians. 

Just as in our calculation of the K-theoretic J-function of projective space itself, we will want to compare $M^X_d \subset G^X_d$ to some $X_d \subset \PP^N_d$. The equations that cut $X$ and $Y$ out of $\PP^N$ naturally extend to cut out $X_d \subset Y_d \subset \PP^N_d$ because $X$ and $Y$ are homogeneous. We will also need the Quot schemes $Q^X_d \subset Q^Y_d$, where $Q^Y_d$ is another smooth compactification of $\Map_d (\PP^1,Y)$. There is an equivariant morphism $v: Q^Y_d \rightarrow \PP^N_d$. A sheaf in the Grothendieck group of $Q^Y_d$ that will enable us to calculate the K-theoretic J-function of $X$ from the K-theoretic J-function of $Y$ is desired; we show that $\OO_{Q^X_d}$ suffices if $Q^X_d$ has rational singularities.


For $Y$ a Grassmannian or product of Grassmannians, $Y_d$ is the closure of $\Map_d (\PP^1,Y) \subset \Map_d(\PP^1, \PP^N) \subset \PP^N_d$. $Y_d$ may be singular, but has rational singularities \cite{SS}. There is a birational map $u:G^Y_d \rightarrow Y_d$ so that in cohomology $u_*[G^Y_d] = [Y_d]$. Since $G^Y_d$ is a rational desingularization we also have $u_*[\OO_{G^Y_d}] = [\OO_{Y_d}]$ in K-theory. To prove the lemma below, assume that there exists  $\OO_{Q^X_d} \in K(Q^Y_d)$ satisfying \begin{equation} (\dagger)\; v_*(\OO_{Q^X_d}) = \OO_d := u_* (\OO_{G^X_d}). \end{equation} This is equivalent to knowing that $Q^X_d$ has rational singularities.


Consider the diagram 

\begin{diagram} 
G^X_d &\rTo_{u} &\PP_d^{N} &\lTo_{v} &Q^Y_d\\
\uInto_{\alpha^X_d}(0,5) & &\uInto_{\alpha_d} & &\uInto_{\alpha_F}(0,5) \\
& &\PP^N && \\
& &\uInto_{j} &\luTo_f & \\ 
M^X_d &\rTo &Y &\lTo_g &F \\ 
&\rdTo_{ev} &\uInto_{k} &\\ 
& &X 
\end{diagram}

\begin{lem}\label{thm: simplelemma} Let $Y \subset \PP^n$ be a homogeneous variety, and $X$ the zero scheme of a regular section of a $T$-equivariant vector bundle $E \in K^T_0(X)$. (Here $T = (\C^*)^n$.) Then \begin{equation}J^{X,K}_d ( q)= k^*g_* \left( \frac{\alpha_F^* [\OO_{Q^X_d}] /g^*\lambda_{-1}^{\C^*} (E^{\vee})}{\lambda_{-1}^{\C^*} (N^{\vee}_{F/Q^Y_d})} \right) \end{equation}\end{lem}

\begin{proof}We use $\C^*$-equivariant K-theory to compute the degree $d$ component of the $J$-function of $X$. By definition, \begin{equation}J_d^{X,K}(Q,q) = ev_* \left( \frac{[\OO_{M^X_d}]}{\lambda_{-1}^{\C^*}(N^{\vee}_{M^X_d/G^X_d})} \right) = ev_* \left( \frac{[\OO_{M^X_d}]}{(1-q \LL)} \right)\end{equation} Recall that $\C^*$ acts on the graph space $G^X_d$ and has $M^X_d$ as a component of the fixed locus. Let $i$ be the composition of inclusions $k: X \hookrightarrow Y$ and $j: Y \hookrightarrow \PP^N$. Apply $i_*$ to both sides:  \begin{equation} i_*J^X_d(q) = i_*ev_* \left( \frac{[\OO_{M^X_d}]}{\lambda_{-1}^{\C^*}(N^{\vee}_{M^X_d/G^X_d})} \right) \end{equation}

From the diagram preceding the lemma, notice that we can use some version of the correspondence of residues, noticing that $ \frac{[\OO_{M^X_d}]}{\lambda_{-1}^{\C^*}(N^{\vee}_{M^X_d/G^X_d})}$ itself is the residue on $M_d^X$ of the class $[\OO_{G^X_d}]$. Thus  \begin{align}  i_*ev_* \left( \frac{[\OO_{M^X_d}]}{\lambda_{-1}^{\C^*}(N^{\vee}_{M^X_d/G^X_d})} \right)&=   \frac{\alpha_d^*u_*[\OO_{G^X_d}]}{\lambda_{-1}^{\C^*}(N^{\vee}_{\PP^N/\PP^N_d})}  \\ &=  \frac{\alpha_d^*v_*[\OO_{Q^X_d}]}{\lambda_{-1}^{\C^*}(N^{\vee}_{\PP^N/\PP^N_d})} . \end{align} The second equality follows from the $(\dagger)$ assumption.

We can use correspondence of residues again with another part of the diagram:
\[\begin{CD}
  Q^Y_d @>v>> \PP^N_d\\
   @A{\alpha_F}AA @A{\alpha_d}AA\\ 
  F @>f>> \PP^N
\end{CD}.\]

Here $F$ is the fixed component of $Q^Y_d$; for ease of notation we pretend there is only one component although there may be several. This gives us  \begin{align}\label{eq: aboveeq}  i_* J^{X,K}_d (q) &= \frac{\alpha_d^*v_*[\OO_{Q^X_d}]}{\lambda_{-1}^{\C^*}(N^{\vee}_{\PP^N/\PP^N_d})} \\ &= f_* \frac{\alpha_F^*[\OO_{Q^X_d}]}{\lambda_{-1}^{\C^*}(N^{\vee}_{F/Q^Y_d})} \\ &= j_*g_*\frac{\alpha_F^*[\OO_{Q^X_d}]}{\lambda_{-1}^{\C^*}(N^{\vee}_{F/Q^Y_d})}. \end{align} Notice that $f$ is a composition of $g$ and $j$.

Since $X$ is cut out of $Y$ as the zero-section of a $T$-equivariant bundle, $\C^* \times T$-equivariant K-theory is used to obtain the equality \begin{equation} J^{X,K}_d (q) = \frac{i^*i_*J^{X,K}_d(q)}{\lambda_{-1}^{\C^*}(N^{\vee}_{X/\PP^N})}.\end{equation} Use (\ref{eq: aboveeq}) to rewrite: \begin{equation} J^{X,K}_d(q) = \frac{1}{\lambda_{-1}^{\C^*}(N^{\vee}_{X/\PP^N})} i^*\left(j_*g_*\frac{\alpha_F^* [\OO_{Q^X_d}]}{\lambda_{-1}^{\C^*}(N^{\vee}_{F/Q^Y_d})}\right). \end{equation} For the last time, we use correspondence of residues noticing that $i$ is the composition of $j$ and $k$.

\begin{equation} J^{X,K}_d (q)=\frac{1}{\lambda_{-1}^{\C^*}(N^{\vee}_{X/\PP^N})} i^*\left(j_*g_*\frac{\alpha_F^* [\OO_{Q^X_d}]}{\lambda_{-1}^{\C^*}(N^{\vee}_{F/Q^Y_d})}\right) =  \frac{1}{\lambda_{-1}^{\C^*}(N^{\vee}_{X/Y})} k^*g_*\frac{\alpha_F^* [\OO_{Q^X_d}]}{\lambda_{-1}^{\C^*}(N^{\vee}_{F/Q^Y_d})}. \end{equation}

Rewrite $\lambda_{-1}^{\C^*}(N^{\vee}_{X/Y}) = k^* \lambda_{-1}^{\C^*} (E^{\vee})$. This is a sum of locally free sheaves, so the projection formula implies \begin{equation} J^X_d (e^{-\hbar}) = k^*g_* \left( \frac{\alpha^*_F [\OO_{Q^X_d}] /g^* \lambda_{-1}^{\C^*} (E^{\vee}) }{\lambda_{-1}^{\C^*} (N^{\vee}_{F/Q^Y_d})} \right) .\end{equation} This will be used to calculate the K-theoretic J-functions of the flag varieties of type A. \end{proof}

\subsection{Flag Varieties of Type A}
 Consider, as in \cite{BCK2}, the embedding of the flag into the product of Grassmannians \begin{equation} Fl(m_1, \ldots, m_{\ell},n) \subset \prod_{i=1}^{\ell} Gr(m_i, n) \subset \PP^N. \end{equation} The flag variety is cut out as the zero scheme of a section of the vector bundle \begin{equation}E = \oplus_{i=1}^{\ell} Hom(S_i, Q_{i+1}) \cong  \oplus_{i=1}^{\ell} S_i^{\vee} \otimes Q_{i+1}, \end{equation} where $S_i$ is the tautological bundle and $Q_i$ the quotient bundle for each $Gr(m_i,n)$. Look at the corresponding product of Quot schemes of vector bundle subsheaves $K \subset \C^n \otimes \OO_{\PP^1}$ of degree $-d_i$ and rank $m_i$. We denote by $HQ_d$ the zero scheme of the section of $\pi_* \left(  \oplus_{i=1}^{\ell} \K^{\vee}_i \otimes V/\K_{i+1} \right)$ on this product. This is the fundamental class of the ``hyperquot'' scheme, and it is smooth and irreducible. Its structure sheaf $\OO_{HQ_d}$ satisfies the $(\dagger)$ condition \cite{GL}.

\begin{thm} The K-theoretic J-function for flag varieties $\Fl(m_1, m_2, \ldots, m_{\ell},S)$ of type A is \begin{displaymath}\label{eq: Aflag} J^{\Fl,K}(Q,q) = \sum_d J^{\Fl,K}_d(q) Q^d\end{displaymath} where \begin{equation}\begin{split} J^{\Fl,K}_d(q)= \sum_{\sum d_{i,j} = d_i}\prod_{i=1}^{\ell} (-1)^{(m_i-1) d_i} \prod_{1 \leq k \neq j \leq m_i}\frac{ \prod^{d_{i,k}-d_{i,j}}_{m= - \infty} (1-\LL^{\vee}_{i,j} \otimes \LL_{i,k} q^m)}{\prod^0_{m=-\infty} (1- \LL^{\vee}_{i,j} \otimes \LL_{i,k}q^m)}  \\ \cdot \prod_{1 \leq j \leq m_i, 1 \leq k \leq m_{i+1}} \frac{\prod_{m= - \infty}^0 (1- \LL_{i,j}^{\vee} \otimes \LL_{i+1,k} q^m)}{\prod_{m= - \infty}^{d_{i,j}-d_{i+1,k}} (1- \LL_{i,j}^{\vee} \otimes \LL_{i+1,k} q^m)} .\end{split}\end{equation} Here $\LL_{i,j}$ are the $j$ Chern line bundles coming from splitting of $S_i^{\vee}$ andthe sum is over $d=(d_1, \ldots, d_{\ell})$, where $d_i$ comes from pairing the curve class with $c_1(S^{\vee}_i)$.\end{thm} 

\begin{proof} Finding $\lambda_{-1}^{\C^*}(E^{\vee}) $ is a fairly straightforward computation using the exact sequence \begin{equation} 0 \rightarrow \oplus_{i=1}^{\ell} S_i^{\vee} \otimes S_{i+1} \rightarrow \oplus_{i=1}^{\ell -1} S_i^{\vee} \otimes \C^n \rightarrow E \rightarrow 0. \end{equation} To compute $\alpha^*_F[\OO_{HQ_d}]$ we exploit the Koszul complex for $\pi_* (\K_i^{\vee} \otimes V / \K_{i+1})$, since $HQ_d$ is cut out of $Q^Y_d$ by a regular section $s$ of this bundle and thus the Koszul complex gives a finite locally free resolution for $\OO_{HQ_d}$. (Recall that $\pi$ is the projection from $\PP^1 \times \prod_i Q^{Gr(m_i,n)}_{d_i}$ to $\prod_i Q^{Gr(m_i,n)}_{d_i}$, and for each Quot scheme $Q^{Gr(m_i,n)}_{d_i}$, $\K_i$ is the universal subsheaf.)

On $\PP^1 \times \prod_i Q^{Gr(m_i,n)}_{d_i} $, we have
 \begin{equation} 0 \rightarrow \K_i^{\vee} \otimes \K_{i+1} \rightarrow \K_i^{\vee} \otimes V \rightarrow \K_i^{\vee} \otimes V/\K_{i+1} \rightarrow 0.  \end{equation} Push forward to the product of Quot schemes, noting that now $\pi_* \K_i^{\vee} \otimes V/ \K_{i+1}$ is an honest vector bundle:  \begin{equation} 0 \rightarrow \pi_*(\K_i^{\vee} \otimes \K_{i+1}) \rightarrow \pi_*(\K_i^{\vee} \otimes V) \rightarrow \pi_*(\K_i^{\vee} \otimes V/\K_{i+1}) \rightarrow R^1\pi_*(\K_i^{\vee} \otimes \K_{i+1}) \rightarrow 0.  \end{equation} Use this exact sequence to explicitly calculate $\lambda_{-1}^{\C^*} (\pi_* \left(  \oplus_{i=1}^{\ell} \K^{\vee}_i \otimes V/\K_{i+1} \right))$, which happens to given the Koszul complex of $\OO_{HQ_d}$. 
From the sequence,
\begin{equation}\lambda_{-1}^{\C^*}  \left( \left(\pi_* \left(  \oplus_{i=1}^{\ell} \K^{\vee}_i \otimes V/\K_{i+1} \right)\right)^{\vee}\right) = \frac{\lambda_{-1}^{\C^*} ( R^1\pi_*(\K_i^{\vee} \otimes \K_{i+1})^{\vee}) \lambda_{-1}^{\C^*}(\pi_*(\K_i^{\vee} \otimes V)^{\vee})}{ \lambda_{-1}^{\C^*}(\pi_*(\K_i^{\vee} \otimes \K_{i+1})^{\vee})}\end{equation}

Now we compute the contribution of each term separately. Computations take place on the fixed locus of $Q^Y_d$, whose components are indexed by splittings of $d$. First,

\begin{equation} \lambda_{-1}^{\C^*} (\pi_*(\oplus_{i=1}^{\ell -1} \K_i^{\vee} \otimes V)^{\vee}) = \prod_{i=1}^{\ell} \prod_{j=1}^{m_i} \prod_{m=0}^{d_{i,j}} (1- \LL^{\vee}_{i,j}q^m)^n \end{equation} 

Look then at $ R^0\pi_*(\oplus_{i=1}^{\ell} \K_i^{\vee} \otimes \K_{i+1})-R^1\pi_*(\oplus_{i=1}^{\ell} \K_i^{\vee} \otimes \K_{i+1})$. Fiberwise, we can look at this as \begin{equation}H^0 (\PP^1, \oplus_{i=1}^{\ell} \oplus_{j,k} \LL_{i,j} \otimes \LL^{\vee}_{i+1,k} (d_{i,j}-d_{i+1,k}) )-H^1(\PP^1, \oplus_{i=1}^{\ell} \oplus_{j,k} \LL_{i,j} \otimes \LL^{\vee}_{i+1,k} (d_{i,j}-d_{i+1,k}) ) .\end{equation} Use Serre duality: \begin{multline} R^{\bullet}\pi_*(\oplus_{i=1}^{\ell} \K_i^{\vee} \otimes \K_{i+1}) = H^0 (\PP^1, \oplus_{i=1}^{\ell} \oplus_{j,k} \LL_{i,j} \otimes \LL^{\vee}_{i+1,k} (d_{i,j}-d_{i+1,k}) )\\ -H^0(\PP^1, \oplus_{i=1}^{\ell} \oplus_{j,k} \LL^{\vee}_{i,j} \otimes \LL_{i+1,k} (-d_{i,j}+d_{i+1,k}+2) )^{\vee}. \end{multline} For $d_{i,j}-d_{i+1,k} \geq 0$ we get a contribution of \begin{equation} \prod_{i=1}^{\ell} \prod_{j,k} \prod_{m=0}^{d_{i,j}-d_{i+1,k}}(1- \LL^{\vee}_{i,j} \otimes \LL_{i+1,k} q^m).\end{equation} For $d_{i,j} - d_{i+1,k} \leq -2$, we get \begin{equation} \prod_{i=1}^{\ell} \prod_{j,k} \prod_{m=1}^{d_{i+1,k}-d_{i,j}-1}(1- \LL^{\vee}_{i,j} \otimes \LL_{i+1,k} q^m). \end{equation}  For $d_{i,j} - d_{i+1,k} =-1$, we get no contribution. We will see that $\lambda_{-1}^{\C^*}(E^{\vee})$ will pair nicely with this, so for now write the contribution as \begin{equation} \frac{\lambda_{-1}^{\C^*} ( R^1\pi_*(\K_i^{\vee} \otimes \K_{i+1})^{\vee})}{ \lambda_{-1}^{\C^*}(\pi_*(\K_i^{\vee} \otimes \K_{i+1})^{\vee})}= \prod_{i=1}^{\ell} \prod_{1 \leq j \leq m_i, 1 \leq k \leq m_{i+1}} \frac{\prod_{m= - \infty}^{-1} (1- \LL_{i,j}^{\vee} \otimes \LL_{i+1,k} q^m)}{\prod_{m= - \infty}^{d_{i,j}-d_{i+1,k}} (1- \LL_{i,j}^{\vee} \otimes \LL_{i+1,k} q^m)} \end{equation}

The contributions of these terms are then combined to produce
\begin{equation} \alpha^*_F ([\OO_{HQ_d}]) = \prod_{i=1}^{\ell} \prod_{1 \leq j \leq m_i,1 \leq k \leq m_{i+1}} \frac{ \prod_{m=0}^{d_{i,j}} (1- \LL^{\vee}_{i,j}q^m)^n}{(1 - \LL^{\vee}_{i,j} \otimes \LL_{i+1,k} q^{d_{i,j}-d_{i+1,k}})(1 - \LL^{\vee}_{i,j} \otimes \LL_{i+1,k})}, \end{equation} 

Since we are trying to compute  $J^{\Fl,K}_d (q)= k^*g_* \left( \frac{\alpha^*_F [\OO_{Q^X_d}] /g^* \lambda_{-1}^{\C^*} (E^{\vee}) }{\lambda_{-1}^{\C^*} (N^{\vee}_{F/Q^Y_d})} \right) $, we still need $\lambda_{-1}^{\C^*} (E^{\vee})$.  \begin{align} \lambda_{-1}^{\C^*} (E^{\vee}) &= \frac{\lambda_{-1}^{\C^*}((\oplus_{i=1}^{\ell -1} S_i^{\vee} \otimes \C^n)^{\vee} )}{\lambda_{-1}^{\C^*}((\oplus_{i=1}^{\ell} S_i^{\vee} \otimes S_{i+1})^{\vee}) }\\  &= \frac{\lambda_{-1}^{\C^*}(\oplus_{i=1}^{\ell -1}  (\oplus_{j=1}^{m_i} \LL_{i,j} )\otimes \C^n)^{\vee}}{\lambda_{-1}^{\C^*}(\oplus_{i=1}^{\ell -1}  (\oplus_{j=1}^{m_i} \LL_{i,j} )\otimes (\oplus_{k=1}^{m_{i+1}} \LL_{i+1,k}^{\vee})^{\vee}} \\ &=\prod_{i=1}^{\ell} \frac{\prod_{1 \leq j \leq m_i} (1- \LL^{\vee}_{i,j})^n}{\prod_{1 \leq j \leq m_i, 1 \leq k \leq m_{i+1}} (1-\LL^{\vee}_{i,j} \otimes \LL_{i+1,k})} \end{align}

Combine to obtain the numerator: \begin{equation}\label{eq: Aflagguts}\frac{\alpha^*_F([\OO_{HQ_d}])}{\lambda_{-1}^{\C^*} (E^{\vee})} =  \prod_{i=1}^{\ell} \prod_{1 \leq j \leq m_i,1 \leq k \leq m_{i+1}} \frac{ \prod_{m=1}^{d_{i,j}} (1- \LL^{\vee}_{i,j}q^m)^n}{(1 - \LL^{\vee}_{i,j} \otimes \LL_{i+1,k} q^{d_{i,j}-d_{i+1,k}})}.\end{equation}

This needs to be pushed forward to the product of Grassmannians via $g$, along with $\lambda_{-1}^{\C^*} ([N^{\vee}_{F/Q^Y_d}])$. See (\ref{thm: pushf}) for the formula for pushforward from a flag variety to a Grassmannian, which can be easily extended to the product of Grassmannians. See (\ref{eq: FinQYd}) for calculation of $\lambda_{-1}^{\C^*}([N^{\vee}_{F/Q^Y_d}])$. After pushing forward to the product of Grassmannians, pull back by $k^*$; since we're dealing with sums of line bundles, $k^*$ is injective.

 The K-theoretic J-function of the flag variety, then, is given by $\sum_d Q^d J_d^{Fl,K}(q)$ and \begin{multline}J_d^{Fl,K}(q) = \prod_{i=1}^{\ell} (-1)^{(m_i-1) d_i} \prod_{1 \leq k < j \leq m_i}\frac{ (1-\LL^{\vee}_{i,j} \otimes \LL_{i,k} q^{d_{ij}-d_{i,k}})}{ (1- \LL^{\vee}_{i,j} \otimes \LL_{i,k})} \cdot \\ \prod_{1 \leq j \leq m_i, 1 \leq k \leq m_{i+1}} \frac{\prod_{m= - \infty}^0 (1- \LL_{i,j}^{\vee} \otimes \LL_{i+1,k} q^m)}{\prod_{m= - \infty}^{d_{i,j}-d_{i+1,k}} (1- \LL_{i,j}^{\vee} \otimes \LL_{i+1,k} q^m)} .\end{multline} This can be rewritten in a form parallel to that in \cite{BCK2, GToric}:\begin{multline}
J_d^{Fl,K}(q) =\prod_{i=1}^{\ell} (-1)^{(m_i-1) d_i} \prod_{1 \leq k \neq j \leq m_i}\frac{ \prod^{d_{i,k}-d_{i,j}}_{m= - \infty} (1-\LL^{\vee}_{i,j} \otimes \LL_{i,k} q^m)}{\prod^0_{m=-\infty} (1- \LL^{\vee}_{i,j} \otimes \LL_{i,k}q^m)} \cdot \\ \prod_{1 \leq j \leq m_i, 1 \leq k \leq m_{i+1}} \frac{\prod_{m= - \infty}^0 (1- \LL_{i,j}^{\vee} \otimes \LL_{i+1,k} q^m)}{\prod_{m= - \infty}^{d_{i,j}-d_{i+1,k}} (1- \LL_{i,j}^{\vee} \otimes \LL_{i+1,k} q^m)} .\end{multline}\end{proof}

\section{Directions for further research}

Similar techniques apply to Lie groups of other types. Unfortunately, a direct application of the method of proof requires rationality of certain subschemes of the Quot scheme, which is currently not known. The following conjectures, however, follow: 

\begin{conj} The K-theoretic J-function for the Lagrangian complete flag variety, denoted by $L\Fl(1,2, \ldots, n, 2n)$, is \begin{displaymath} J^{L \Fl,K}(Q,q) = \sum_d J^{L \Fl,K}_d(q) Q^d\end{displaymath} where \begin{multline}J^{L\Fl,K}_d = \sum \prod_{i=1}^{n}(-1)^{(i-1)d} \left( \frac{\prod_{1 \leq j<k \leq i} \prod_{m=0}^{d_{j}+d_{k}} (1-\LL^{\vee}_{j} \otimes \LL_{k}^{\vee} q^m)}{\prod_{1 \leq j<k \leq i}  (1-\LL^{\vee}_{j} \otimes \LL_{k}^{\vee} )} \right) \\ \cdot \left( \prod_{1 \leq j < k \leq i} \frac{ (1-\LL_{k}^{\vee}\otimes \LL_{j} q^{d_{k}-d_{j}})}{(1-\LL_{k}^{\vee} \otimes \LL_{j})} \right)\\ \left( \prod_{i=1}^{n-1} \prod_{1 \leq j \leq i,1 \leq k \leq i+1} \frac{1}{(1 - \LL^{\vee}_{j} \otimes \LL_{k} q^{d_{j}-d_{k}})} \right). \end{multline} Here $\LL_i = S^{\vee}_i/S^{\vee}_{i-1}$ and $d_i$ is obtained by pairing the curve class with $\LL_i$. \end{conj}

\begin{conj} The K-theoretic J-function for the complete flag varieties of types $B$ and $D$ is \begin{displaymath} J^{I \Fl,K}(Q,q) = \sum_d J^{I \Fl,K}_d(q) Q^d\end{displaymath} where \begin{multline}J^{I\Fl,K}_d = \sum_{d_1 + \cdots +d_{n} = d} \prod_{i=1}^{n}(-1)^{(i-1)d} \left( \frac{\prod_{1 \leq j \leq k \leq i} \prod_{m=0}^{d_{j}+d_{k}} (1-\LL^{\vee}_j \otimes \LL_{k}^{\vee} q^m)}{\prod_{1 \leq j \leq k \leq i}  (1-\LL^{\vee}_{j} \otimes \LL_{k}^{\vee} )} \right) \\ \cdot \left( \prod_{1 \leq j < k \leq i} \frac{ (1-\LL_{k}^{\vee}\otimes \LL_{j} q^{d_{k}-d_{j}})}{(1-\LL_{k}^{\vee} \otimes \LL_{j})} \right)\\ \left( \prod_{i=1}^{n-1} \prod_{1 \leq j \leq i,1 \leq k \leq i+1} \frac{1}{(1 - \LL^{\vee}_{j} \otimes \LL_{k} q^{d_{j}-d_{k}})} \right). \end{multline}\end{conj}

Another natural problem is establishing K-theoretic J-functions for toric varieties, and comparing the K-theoretic J-functions of flag and toric varieties to see whether the abelian-nonabelian correspondence also extends to that situation.

\bibliographystyle{amsalpha}

\bibliography{bigbibliography}

\end{document}